\documentclass[final,3p,mathptmx]{article}

\usepackage{graphicx}
\usepackage{amsmath}
\usepackage{amsfonts}
\usepackage{amssymb}
\usepackage{amsthm}
\usepackage{newlfont}
\usepackage{longtable}
\usepackage{verbatim}

\usepackage{color}
\usepackage{tikz}
\usepackage{mathrsfs}

\newcommand{\R}{{\mathbb R}}
\newcommand{\N}{{\mathbb N}}

\newcommand{\Rd}{{\R}^d}
\newcommand{\Comp}{\mathrm{K}(\Rd)}

\newcommand{\inK}{\in\Comp}

\newcommand{\haus}{\mathrm{haus}}
\newcommand{\dist}{\mathrm{dist}}

\newcommand{\CH}{{\mathrm{CH}}}
\newcommand{\Graph}{{\mathrm{Graph}}}

\newcommand{\sign}{{\mathrm{sign}}}

\newcommand{\PrjXonY}[2]{\Pi_{#2}{(#1)} }
\newcommand{\Pair}[2]{\Pi \big ( {#1},{#2} \big )}

\newcommand{\WeightedMetInt} { { \scriptstyle( \cal M_{\kappa} ) } \int_{a}^{b} \kappa(x)F(x)dx }

\newcommand{\Map}[1]{F:#1 \rightarrow \Comp }

\newcommand{\LipAtPoint}[2]{ \mathrm{Lip} \{ {#1}, {#2} \} }

\newcommand{\CBV}{\mathrm{CBV}[a,b]}
\newcommand{\BV}{\mathrm{BV}[a,b]}

\newcommand{\calF}{\cal{F}[a,b]}

\newcommand{\LeftLocalModul}[3]{\omega^{-}\big( {#1},{#2},{#3} \big)}

\newcommand{\NewLeftLocalModul}[3]{\varpi^{-}\big( {#1},{#2},{#3} \big)}
\newcommand{\NewRightLocalModul}[3]{\varpi^{+}\big( {#1},{#2},{#3} \big)}
\newcommand{\QuasiLocalModul}[3]{\varpi\big( {#1},{#2},{#3} \big)}
\newcommand{\LocalModulCont}[3]{\omega\big( {#1},{#2},{#3} \big)}

\newcommand{\ModulContPOne}[2]{\vartheta \big( {#1},{#2} \big)}
\newcommand{\setMS}{\mathcal{S}(F)}

\newcommand{\SymbText}[1]{\hbox to 10cm { \dotfill #1 \dotfill}}

\newtheorem{remark}{Remark}[section]

\newtheorem{theorem}[remark]{Theorem}

\newtheorem{corol}[remark]{Corollary}
\newtheorem{lemma}[remark]{Lemma}
\newtheorem{result}[remark]{Result}

\newtheorem{defin}[remark]{Definition}

\begin{document}

\title {Metric approximation of set-valued functions of bounded variation by integral operators}

\renewcommand{\thefootnote}{\fnsymbol{footnote}}
\author{
Elena E. Berdysheva \footnotemark[1], \quad
Nira Dyn \footnotemark[2], \quad
Elza Farkhi \footnotemark[2]\;\quad
Alona Mokhov \footnotemark[4]\;
}
\footnotetext[1]{ University of Cape Town, South Africa}
\footnotetext[2]{Tel-Aviv University, School of Mathematical Sciences}
\footnotetext[4]{Afeka, Tel-Aviv Academic College of Engineering}
\date{}
\maketitle

\medskip

{ \small {\normalsize \textbf {Abstract.} We introduce an adaptation of integral approximation operators to set-valued functions (SVFs, multifunctions), mapping a compact interval $[a,b]$ into the space of compact non-empty subsets of $\R^d$.	All operators are adapted by replacing the Riemann integral for real-valued functions by the weighted metric integral for SVFs of bounded variation with compact graphs. For such a set-valued function $F$, we obtain pointwise error estimates for sequences of integral operators at points of continuity, leading to convergence at such points to $F$.
At points of discontinuity of $F$, we derive estimates, which yield the convergence to a set, first described in our previous work on the metric Fourier operator. 
Our analysis uses recently defined one-sided local quasi-moduli at points of discontinuity and several notions of local Lipschitz property at points of continuity. 

We also provide a global approach for error bounds.  A multifunction $F$ is represented by the set of all its metric selections, while its approximation (its image under the operator) is represented by the set of images of these metric selections   under the operator.
A bound on the Hausdorff distance between these two sets of single-valued functions in $L^1$ provides our global estimates.

The theory is illustrated by presenting the examples of two concrete operators: the Bernstein-Durrmeyer operator and the Kantorovich operator.		
	
\medskip

\noindent{ \small {\normalsize \textbf{Key words:}} { Set-valued functions, functions of bounded variation, metric integral, metric approximation, integral operators, positive linear operators, rate of convergence }

\noindent{ \small {\normalsize \textbf{Mathematics Subject Classification 2020:}} 26E25, 28B20, 41A35, 41A36, 41A25, 26A45


\section {Introduction} \label{Sect_Intro}

We study set-valued functions (SVFs, multifunctions) that map a compact interval $[a,b] \subset \R$ into the space of compact non-empty subsets of $\R^d$. 
These functions  appear in different fields such as dynamical systems, control theory, optimization, game theory, differential inclusions, economy, geometric modeling. See the book~\cite{AubinFrankowska:90} for foundations of set-valued analysis.

Approximation methods for SVFs has been developing in the last decades. 
Older works deal mostly with convex-valued multifunctions and their approximation based on Minkowski linear combinations,  e.g., \cite{MNikolskii:Opt90, DynFarkhi:00, DF:04, Muresan:SVApprox2010, BaierPerria:11, Campiti:2019}. In~\cite{VIT:79}, such an adapation of the classical Bernstein polynomial operator is proved to converge to SVFs with convex compact images (values).
Yet, it is shown that this adaptation fails to approximate general SVFs (with general compact not necessarily convex images). 
In general, approximation methods developed for multifunctions with convex images usually are not suitable for general SVFs. 
 

A first successful attempt to approximate general SVFs from their samples is accomplished by Z.~Artstein in~\cite{Artstein:MA}, where piecewise linear approximants are constructed. This is done by replacing binary Minkowski average of two sets with the metric average, which is further extended in~\cite{DFM:Chains} to the metric linear combination of several sets. 
Based on the metric linear combination, N.~Dyn, E.~Farkhi and A.~Mokhov developed in a series of works~\cite{DynFarkhi:01, DynMokhov:06, DFM:Chains, DFM:07serdica, DFM:Book_SV-Approx} adaptation of classical sample-based approximation operators to continuous general SVFs. For these adapted operators, termed metric operators, error estimates are obtained, which for most operators are similar to those obtained in the real-valued case. Special attention is given to Bernstein polynomial operators, Schoenberg spline operators and polynomial interpolation operators. Later in~\cite{	BDFM:2019}, the above metric approach is extended to SVFs of bounded variation.

The metric approach is applied in~\cite{DFM:MetricIntegral} to introduce and study the metric integral for general SVFs of bounded variation. The metric integral is not necessarily convex in contrast to the Aumann integral, which is always convex, even if the integrand is not convex-valued~\cite{AUM:65}.  In~\cite{BDFM:2021} the metric integral is extended to the weighted metric integral, which is used, with the Dirichlet kernels as weight functions, to define metric Fourier partial sums for SVFs of bounded variation. The convergence of these partial sums is analyzed at points of continuity of a multifunction as well as at points of discontinuity.
An important tool in the analysis at points of discontinuity is the notion of one-sided local quasi-moduli of a function of bounded variation.

In this paper we adapt integral approximation operators for real-valued functions to general SVFs of bounded variation. Previous adaptations of integral operators to SVFs are limited to convex-valued multifunctions and are based on the Aumann integral (see e.g.~\cite{Babenko:2016, {Campiti:2022}}).
Our adaptation is based on the weighted metric integral, and its analysis applies and extends the techniques developed in~\cite{BDFM:2021}.

The outline of the paper is as follows.
Section~\ref{Sect_Prelim} gives a short overview of notions we use in the paper, and also discusses different regularity properties of functions with values in a metric space. In Section~\ref{Sect_Convergence_RealValued} we refine known results concerning approximation of real-valued functions by sequences of integral approximation operators, which are necessary for the adaptation of these operators to SVFs. 

  The core part of the paper is Section~\ref{Sect_ApproxSVF} where we construct an adaptation of integral approximation operators to general SVFs. For set-valued functions of bounded variation with compact graphs,  we study pointwise convergence, in the Hausdorff metric, of sequences of such operators at points of continuity of the function as well as at points of discontinuity, and derive estimates for the rate of convergence. In Section~\ref{Sect_SpecificOperators} we illustrate our theory by considering examples  of two particular integral approximation operators, the Bernstein-Durrmeyer operator and the Kantorovich operator.

In the final Section~\ref{Sect_TwoSets} we provide global error bounds. The multifunction $F$ is represented by the set of all its metric selections (see~\cite{DFM:Book_SV-Approx} for more information on representations of SVFs), while its approximation (its image under the operator) is represented by the set of images of these metric selections under the operator.
A bound of the Hausdorff distance between these two sets of single-valued functions in $L^1$ is obtained using results from~\cite{Berens_DeVore}.


\section {Preliminaries}\label{Sect_Prelim}

In this section we introduce some notation and basic notions related to sets and set-valued functions. 
We discuss notions of regularity of functions in metric spaces. We review the notions of metric selections and the metric integral of set-valued functions.

\subsection {On sets}\label{Sect_Prelim_OnSets}

All sets considered from now on are sets in $\Rd$.  We denote by $\Comp$\label{CompSets} the collection of all compact non-empty subsets of~$\Rd$. 
The metric in $\Rd$ is of the form $\rho(u,v)=|u-v|$, where $|\cdot|$ is any fixed norm on $\Rd$. Recall that~$\Rd$ endowed with this metric is a complete metric space and that all norms on $\Rd$ are equivalent.

\noindent To measure the distance between two non-empty sets ${A,B \in \Comp}$, we use the Hausdorff metric  based on~$\rho$
$$
	\haus(A,B)_{\rho}= \max \left\{ \sup_{a \in A}\dist(a,B)_{\rho},\; \sup_{b \in B}\dist(b,A)_{\rho} \right\},
$$
where the distance from a point $c$ to a set $D$ is $\dist(c,D)_{\rho}=\inf_{d \in D}\rho(c,d)$.  
It is well known that $\Comp$ endowed with the Hausdorff metric is a complete metric space~\cite{RockWets, {S:93}}.

In the following, we keep the metric in $\Rd$ fixed, and omit the notation $\rho$ as a subscript.

 We denote by $|A|=\haus (A,\{0\})$ the ``norm'' of the set $A \in \Comp$.
The set of projections of $a \in \Rd$ on a set $B \inK$ is
$
\PrjXonY{a}{B}=\{b \in B \ : \ |a-b|=\dist(a,B)\},
$
and the set of metric pairs of two sets $A,B \inK$ is
$$
\Pair{A}{B} = \{(a,b) \in A \times B \ : \ a \in \PrjXonY{b}{A}\;\, \mbox{or}\;\, b\in\PrjXonY{a}{B} \}.
$$
Using metric pairs, we can rewrite 
$$
\haus(A,B)= \max \{|a-b| \ :\ (a,b)\in \Pair{A}{B}\}.
$$

We recall the notions of  a metric chain and of a metric linear combination~\cite{DFM:MetricIntegral}.
\begin{defin}\label{Def_MetChain_MetSelection} \cite{DFM:MetricIntegral}
	Given a finite sequence of sets $A_0, \ldots, A_n \in \Comp$, $n \ge 1$,  a metric chain of $A_0, \ldots, A_n$ is an $(n+1)$-tuple $(a_0,\ldots,a_n)$ such that $(a_i,a_{i+1}) \in \Pair {A_i}{A_{i+1}}$, $i=0,1,\ldots,n-1$. The collection of all metric chains of $A_0, \ldots, A_n$ is denoted by  
	$$
	\CH(A_0,\ldots,A_n)= \left\{ (a_0,\ldots,a_n) \ : \ (a_i,a_{i+1}) \in \Pair {A_i}{A_{i+1}}, \  i=0,1,\ldots,n-1 \right\}.
	$$
	The metric linear combination of the sets $A_0, \ldots, A_n \in \Comp$, $n \ge 1$, is
	$$
	\bigoplus_{i=0}^n \lambda_i A_i =
	\left\{ \sum_{i=0}^n \lambda_i a_i \ : \ (a_0,\ldots,a_n) \in \CH(A_0,\ldots,A_n) \right\}, \quad \lambda_0,\ldots,\lambda_n \in \R.
	$$
\end{defin}

\begin{remark}\label{Remar_ThereIsChain}
	For any $j\in \N$,  $0 \le j\le n$ and for any $a \in A_j$ there exists a metric chain $(a_0, \ldots, a_n) \in \CH(A_0,\ldots,A_n)$ such that $a_j=a$. For a possible construction see~\cite{DFM:Chains}, Figure~3.2.
\end{remark}

Note that the metric linear combination depends on the order of the sets, in contrast to the Minkowski linear combination of sets which is defined by
$$
\sum_{i=0}^n \lambda_i A_i =
\left\{ \sum_{i=0}^n \lambda_i a_i \ : \  a_i \in A_i \right\},\quad n \ge 1.
$$
The upper Kuratowski limit of a sequence of sets $\{A_n\}_{n=1}^{\infty}$  is the set of all limit points of converging subsequences
$\{a_{n_k}\}_{k=1}^{\infty}$, where ${a_{n_k} \in A_{n_k} }$, $k\in \N$, namely
\begin{equation}\label{def_UpperKuratLim}
\limsup_{n \to \infty} A_n = \left\{a \ : \  \exists\, \{n_k\}_{k=1}^{\infty},\, n_{k+1}>n_k,\, k\in \N,\ \exists \, a_{n_k} \in A_{n_k} \text{ such that } \lim_{k \to \infty}a_{n_k} = a \right\}.
\end{equation}

\subsection {Notions of regularity of functions with values in a metric space}\label{Sect_Prelim_Regularuty}

In this paper we consider functions defined on a fixed compact interval $[a,b] \subset \R$ with values in a complete metric space $(X,\rho)$, where $X$ is either $\Rd$ or $\Comp$. 

We recall the notion of the variation of ${f:[a,b]\rightarrow X}$. Let ${\, \chi=\{x_0,\ldots, x_n\} }$, $a=x_0 < x_1 < \cdots <x_n=b$, be a partition of the interval $[a,b]$ with the norm 
$$
{\displaystyle |\chi|=\max_{0\le i\le n-1} (x_{i+1}-x_i)}.
$$ 
The variation of $f$ on the partition $\chi$ is defined as 
$
V(f,\chi) = \sum_{i=1}^{n} \rho(f(x_i),f(x_{i-1}))\, .
$
The total variation of $f$ on $[a,b]$ is 
$$
V_{a}^{b}(f) = \sup_{\chi} V(f,\chi),
$$
where the supremum is taken over all partitions of $[a,b]$.

A function $f$ is said to be of bounded variation on $[a,b]$ if ${ V_{a}^{b}(f) < \infty}$. We call functions of bounded variation BV functions and write $f \in \BV$. If $f$ is also continuous,  we write $f\in \CBV$.

For $f \in \BV$ the  function $v_f:[a,b]\rightarrow \R$,\, $v_f(x)=V_{a}^{x}(f)$ is called the variation function of $f$. Note that 
$$
V_{z}^{x}(f)=v_f(x)-v_f(z) \quad  \mbox{for} \quad a\le z<x \le b,
$$
and that $v_f$ is monotone non-decreasing.\\

\noindent For a BV function ${f:\R \rightarrow X}$ the following property holds (see e.g. Lemma 2.4 in~\cite{BDFM:2019}),
\begin{equation}\label{IntegralOfVariation}
\int_{a}^{b} V_{x-\delta}^{x+\delta}(f)dx \le 2\delta V_{a}^{b}(f).	
\end{equation}

\noindent We recall the notion of the local modulus of continuity~\cite{SendovPopov},  which is central to the approximation of functions at continuity points.

For $f : [a,b] \to X$  the local modulus of continuity at $x^* \in [a,b]$ is
$$
	\LocalModulCont{f}{x^*}{\delta} = \sup \left\{\, \rho(f(x_1),f(x_2)): \; x_1,x_2 \in \left[x^*-\delta/2,x^*+\delta/2 \right]\cap[a,b] \,\right\},\quad \delta >0.
$$

It follows from the definition of the variation that~
\begin{result}\label{Result_f_cont->v_f_cont}
	For a function $f : [a, b] \to X$, $f\in \BV$,
			$$\omega(f,x^*,\delta) \le \omega(v_f,x^*,\delta), \quad x^* \in [a,b], \quad \delta > 0.$$
Moreover, $f$ is continuous at $x^* \in [a,b]$ if and only if $v_f$ is continuous at~$x^*$.
\end{result}

\begin{result}\label{Result_f_OneSidedCont->v_f_OneSidedCont}
		A function $f:[a,b]\to X$, $f \in \BV$ is left continuous at $x^*\in (a,b]$ if and only if $v_f$ is left continuous at~$x^*$.
		The function $f$ is right continuous at $x^* \in [a,b)$ if and only if $v_f$ is right continuous at~$x^*$.
\end{result}

A function $f : [a,b] \to X$ of bounded variation with values in a complete metric space $(X,\rho)$ is not necessarily continuous, but has right and left limits at any point $x$~\cite{Chistyakov:On_BV-mappings}. We denote the one-sided limits by
$$ 
f(x+) = \lim_{ t \to x+0 } f(t), \quad f(x-) = \lim_{ t \to x-0 } f(t) .
$$  

In~\cite{BDFM:2021}, we introduced the notion of the left and right local quasi-moduli.
For a function $f : [a,b] \to X$ of bounded variation, the left local quasi-modulus at point $x^*$ is
\begin{equation}\label{Left_QuasiModuli}
	\NewLeftLocalModul{f}{x^*}{\delta} = \sup{ \big \{ \rho(f(x^*-),f(x)) \ : \ x \in [x^*-\delta,x^*) \cap [a,b] \big\} }, \quad \delta >0\ , \; x^* \in (a,b].
\end{equation}
Similarly, the right local quasi-modulus is
\begin{equation}\label{Right_QuasiModuli}
	\NewRightLocalModul{f}{x^*}{\delta} = \sup{ \{ \rho(f(x^*+),f(x)) \ : \ x \in (x^*,x^* + \delta] \cap [a,b] \} }, \quad \delta >0\ , \; x^* \in [a,b).
\end{equation}
Clearly, for $f \in \BV$ the local quasi-moduli satisfy
$$
	\lim_{ \delta \to 0^+ } \NewLeftLocalModul{f}{x^*}{\delta}=0, \quad x^* \in (a,b],
	\quad \text{and} \quad
	\lim_{ \delta \to 0^+ } \NewRightLocalModul{f}{x^*}{\delta} =0 , \quad x^* \in [a,b) .
$$
Defining
$$ 
\QuasiLocalModul{f}{x^*}{\delta} = \max \{ \NewLeftLocalModul{f}{x^*}{\delta},  \NewRightLocalModul{f}{x^*}{\delta} \},
$$
we obtain
\begin{equation}\label{prop_symm_quasi-modul}
\lim_{ \delta \to 0^+ } \QuasiLocalModul{f}{x^*}{\delta} = 0.
\end{equation}

\begin{lemma}\label{Lemma_QuasiMod_f<=QuasiMod_v_f}
	Let $f : [a,b] \to X$, $f \in \BV$, then for any $x^*\in (a,b]$ or $[a,b)$, respectively,  and $ \delta>0$ we have
	$$
	\NewLeftLocalModul{f}{x^*}{\delta} \le \NewLeftLocalModul{v_f}{x^*}{\delta}, \quad  \NewRightLocalModul{f}{x^*}{\delta} \le \NewRightLocalModul{v_f}{x^*}{\delta}.
	$$
\end{lemma}
\begin{proof}
The first inequality follows from the fact that for $x^*\in (a,b]$ and $\max\{x^*-\delta, a\}  \le x < x^*$
we have
$$
	\rho(f(x^*-), f(x)) \le  \lim_{ t \to x^*-0 }  V_x^t(f) = \lim_{ t \to x^*-0 } v_f(t) - v_f(x) = \rho(v_f(x^*-), v_f(x)) .
$$
Similarly one can show the second inequality.

\end{proof}
\smallskip	

Below we discuss several notions of Lipschitz regularity. A functions $f : [a,b] \to (X,\rho)$  is Lipschitz continuous with a constant $\cal L > 0$, if
$$ 
\rho( f(x), f(y)) \leq {\cal L} |x-y|, \quad \forall\,x,y \in [a,b]. 
$$
 
\begin{defin}\label{Def_Lip}
	Let $f : [a,b] \to (X,\rho)$.
	\begin{itemize} 
		\item[(a)]
		We say that $f$ is \textbf{locally Lipschitz around a point} $x$ with the Lipschitz constant $\cal{L}>0$ if there exists $\delta >0$ such that
		\begin{equation}\label{Lip_functions}
			\rho( f(x_1), f(x_2)) \leq {\cal L} |x_1-x_2|, \quad \forall \, x_1, x_2 \in \left(x-\delta/2, x+\delta/2 \right)  \cap [a,b].
		\end{equation}
		\item[(b)]
		A function $f$ is \textbf{locally Lipschitz at a point} $x$ with the Lipschitz constant $\cal{L}>0$ if there exists $\delta >0$ such that
		\begin{equation}\label{loc_Lip_at_point}
			\rho( f(z), f(x)) \leq {\cal L} |z-x|, \quad \forall\,z \in \left (x-\delta/2, x+\delta/2 \right) \cap [a,b]. 
		\end{equation}
		We denote by $\LipAtPoint { x }{ \cal{L}} $ the collection of all functions $f$ satisfying~\eqref{loc_Lip_at_point}.
		\item[(c)] A function $f$ is \textbf{globally Lipschitz at a point} $x$ 
		if there exists $\cal{L}>0$ such that
		\begin{equation}\label{glob_Lip_at_point}
			\rho( f(z), f(x)) \leq {\cal L} |z-x|, \quad \forall\,z \in  [a,b]. 
		\end{equation}
	\end{itemize}
\end{defin}

\begin{remark}\label{f_AroundLip_then_LocalLIp}
	Note that if $f$ is locally Lipschitz around a point $x$ with the Lipschitz constant $\cal{L}$, then $f \in  \LipAtPoint {x}{ \cal{L} }$, but the inverse implication does not hold. For example, the function $f(x)=x \sin(1/x)$ for $x\neq 0$ and $f(0)=0$ is not locally Lipschitz around $x=0$, while $f \in \LipAtPoint{0}{1}$. 
\end{remark}

\noindent We say that $f : [a,b] \to X$  is \textbf{bounded} on $[a,b]$ if there exists $y^* \in X$ such that 
$$
M(f,y^*) = \sup_{x\in [a,b]} \rho (f(x), y^*) < \infty.
$$

\noindent The following lemmas deal with relations between the above notions.

\begin{lemma}\label{lemma:loc_Lip_glob_Lip_at_point}

		If $f \in \LipAtPoint {x}{ \cal{L} }$ and $f$ is bounded on $[a,b]$, then $f$ is also globally Lipschitz at~$x$. 
\end{lemma}
\begin{proof}
By Definition~\ref{Def_Lip}~(b), there is $\delta>0$ such that for $|z - x| < \frac{\delta}{2}$ we have $\rho( f(z), f(x)) \leq  \mathcal{L} |z - x|$. Also, since $f$ is bounded there is $y^*$ such that
	$$
	\rho( f(z), f(x)) \le \rho( f(z), y^*) + \rho( y^*, f(x)) \le 2M(f,y^*).
	$$
	If $|z - x| \ge \frac{\delta}{2}$ we obtain from the above inequality the estimate $\rho( f(z), f(x)) \le \frac{4 M(f,y^*)}{\delta} |z-x|$.
	
\noindent Altogether \eqref{glob_Lip_at_point} holds with $\widetilde{\mathcal{L}} = \max \left( \mathcal{L}, \frac{4M(f,y^*)}{\delta} \right)$.
\end{proof}

\begin{lemma}\label{Lip_1}
	If $f \in \BV$ is locally Lipschitz around a point $x$ with a constant $\cal L >0$, then
	\begin{enumerate}
		\item [(i)]  $v_f$ is locally Lipschitz around the point $x$ with the same constant $\cal L$,
		\item [(ii)]  $v_f \in  \LipAtPoint{x}{ \cal{L} }$.
	\end{enumerate}
\end{lemma}

\begin{proof} 
	Assume that $f$ satisfies \eqref{Lip_functions} with some $\delta > 0$. 
	To prove (i), let $x- \frac{\delta}{2} < y < z < x+\frac{\delta}{2} $, then  
	$$
	v_f(z)-v_f(y) = V_{y}^{z} (f) 
	= \sup_{\chi} \sum_{j=0}^{n-1} \rho (f(x_{j}), f(x_{j+1}) ) 
	\le \sup_{\chi} \sum_{j=0}^{n-1} \cal{L}|x_{j+1}-x_j| =\cal{L}|y-z|,
	$$
	where the supremum is taken over all partitions ${\chi = \{y=x_0 < \cdots <x_n=z \} }$ of $[y,z]$.  

\noindent By (i) and Remark~\ref{f_AroundLip_then_LocalLIp} we obtain (ii).
\end{proof}

\begin{lemma}\label{Lip_2}
	Let $f \in \BV$.
\begin{enumerate}
	\item [(i)] If $v_f$ is locally Lipschitz around a point $x$ with some $\cal L >0$ and $\delta >0$, then $f$ is locally Lipschitz around $x$ with the same $\cal L$ and $\delta$.
	\item [(ii)] If $v_f \in \LipAtPoint{x}{ \cal{L} }$, then $f \in \LipAtPoint{x}{ \cal{L} }$. 
\end{enumerate}
\end{lemma}
\begin{proof} 
	(i) 
	For all $z,y$ such that ${x- \frac{\delta}{2} < y < z < x+\frac{\delta}{2} }$ we have
\begin{equation}\label{number_1}
	\rho (f(z), f(y)) \le V_{y}^{z} (f) = v_f(z)- v_f(y) \le \cal{L} |z-y|\,. 
\end{equation}
To prove (ii) we replace either $y$ or $z$  in~\eqref{number_1} by $x$.

\end{proof}

\begin{remark}
	The inverse implication of Lemma~\ref{Lip_2} (ii) does not hold. Consider, for example, the function $f(x) = x\sin(1/x)$ for $x\neq 0$ and $f(0)=0$.  
	Clearly, $f \in \LipAtPoint{0}{ \cal{L} }$ with $\cal{L}=1$, but $v_f \notin \LipAtPoint{0}{ \cal{L} }$ for any $\cal{L} >0$.
\end{remark}

\noindent In the following we consider ${f:[a,b]\rightarrow \Rd}$,  $f = \begin{pmatrix} f_1\\ \vdots \\ f_d \end{pmatrix}$. We recall that $|\cdot|$ is a fixed norm on~$\Rd$. 

\begin{lemma}\label{lemma_IntegralProp}
 Let ${f:[a,b]\rightarrow \Rd}$ be Riemann integrable. Then 
$$
\left | \int_a^b f(x)dx \right | \le  \int_a^b | f(x) |dx .
 $$
\end{lemma}
\begin{proof}

Consider a sequence of partitions $\chi_n = \{a=x_0^n < \cdots <x_n^n =b \}$ with $\displaystyle \lim_{n \to \infty }|\chi_n|=0$. Take some  $c_i^n \in [x_{i-1}^n, x_i^n]$,  $i=1,\ldots,n$. Using the triangle inequality in the estimate for the Riemann sums we obtain
$$
	\left | \int_a^b f(x)dx \right | =  \left | \lim_{n \to \infty } \sum_{i = 1}^n f(c_i^n)(x_i^n-x_{i-1}^n)\right | 
	\le  \lim_{n \to \infty } \sum_{i = 1}^n | f(c_i^n)| (x_i^n-x_{i-1}^n) =  \int_a^b |f(x)|dx .
$$

\end{proof}

We recall the notion of integral moduli  of continuity for functions with values in $\Rd$.
We extend a function ${f:[a,b]\rightarrow \Rd}$ outside $[a,b]$ in a simple way preserving its variation on $[a,b]$
$$
	f(x)= \left \{ \begin{array}{ll}
		f(a), & x<a, \\
		f(x), & a \le x \le b, \\
		f(b), & x>b.
	\end{array}
	\right.
$$
For ${f:[a,b]\rightarrow \Rd}$, $f = \begin{pmatrix} f_1\\ \vdots \\ f_d \end{pmatrix}$,  we write  $f \in L^1[a,b]$ if all its components $f_i$, $i=1,\ldots,d$, are in $ L^1[a,b]$. 

\noindent The distance between two functions $f,g \in L^1[a,b]$ is given by
$$ 
\|f-g\|_{L^1}=\int_a^b |f(x)-g(x)|dx.
$$
We denote $ \displaystyle \|f \|_{\infty} = \sup_{x\in[a,b]} |f(x)|$.

The first order integral modulus of continuity of ${f:[a,b]\rightarrow \Rd}$, $f \in L^1[a,b]$ is 
$$
\vartheta_{}(f, \delta)= \vartheta_{1}(f, \delta) = \sup_{0<h\le \delta} \int_{a}^{b} |f(x+h)-f(x) | dx ,
\quad \delta > 0.
$$
The second order integral modulus of continuity is 
$$
\vartheta_{2}(f, \delta) = \sup_{0<h\le \delta}  \int_{a}^{b} |f(x+h)-2f(x)+f(x-h) | dx ,
\quad \delta > 0.
$$
It is easy to see that
\begin{equation}\label{relation_moduli}
	\vartheta_{2}(f, \delta) \le 2  \vartheta_{}(f, \delta) .
\end{equation}
Using~\eqref{IntegralOfVariation} one can easily obtain that for $f \in \BV$
\begin{equation}\label{estimate_modulu_p=1}
	\ModulContPOne{f}{\delta} \le \delta V_a^b(f) .
\end{equation}

\subsection {Metric selections and the weighted metric integral of multifunctions}\label{Sect_MS_MetIntegral}
We consider set-valued functions (SVFs, multifunctions) mapping a compact interval $[a,b] \subset \R$ to $\Comp$. 
The graph of a multifunction $F$ is the set of points in $\R^{d+1}$ 
$$
\Graph(F)= \left \{(x,y) \ : \ y\in F(x),\; x \in [a,b] \right \}.
$$
It is easy to see that if $F \in  \BV$ then $\Graph(F)$ is a bounded set and 
$$
\|F\|_\infty =  \sup_{x\in[a,b]} | F(x) | <  \infty . 
$$
We denote the class of SVFs of bounded variation with compact graphs by~$\calF$. 



For a set-valued function $F : [a,b] \to \Comp$, a single-valued function ${s:[a,b] \to \Rd}$ such that $s(x) \in F(x)$ for all $x \in [a,b]$ is called a selection of~$F$.

The notions of the metric selections and of the weighted metric integral are central in our work. We recall their definitions. 

Given a multifunction $F: [a,b] \to \Comp$, a partition $\chi=\{x_0,\ldots,x_n\} \subset [a,b]$, $a=x_0 < \cdots < x_n=b$, and a corresponding metric chain $\phi=(y_0,\ldots,y_n) \in \CH \left ( F(x_0),\ldots,F(x_n) \right )$ (see Definition~\ref{Def_MetChain_MetSelection}), the \textbf{chain function} based on $\chi$ and $\phi$ is
\begin{equation}\label{def_ChainFunc}
	c_{\chi, \phi}(x)= \left \{ \begin{array}{ll}
		y_i, & x \in [x_i,x_{i+1}), \quad i=0,\ldots,n-1,\, \\
		y_n, & x=x_n.
	\end{array}
	\right.
\end{equation}

A selection $s$ of $F$ is called a \textbf{metric selection}, if there is a sequence of chain functions $\{ c_{\chi_k, \phi_k} \}_{k \in \N}$ of~$F$ with ${\lim_{k \to \infty} |\chi_k| =0}$ such that
$$
s(x)=\lim_{k\to \infty} c_{\chi_k, \phi_k}(x) \quad \mbox{pointwisely on} \ [a,b].
$$
We denote the set of all metric selections of $F$ by $\setMS$.

Note that the definitions of chain functions and metric selections imply that a metric selection $s$ of a multifunction $F$ is constant in any open interval where the graph of $s$ stays in the interior of  $\Graph(F)$.

Below we quote some results from~\cite{DFM:MetricIntegral} and~\cite{BDFM:2021} which are used in this paper.

	\begin{result}\label{Result_MetrSel_InheritVariation}\cite[Theorem~3.6]{DFM:MetricIntegral}
		
		Let $s$ be a metric selection of $F \in \calF$. Then $V_a^b(s) \le V_a^b(F)$ and $\|s\|_\infty\le \|F\|_\infty$.
	\end{result}
	\begin{result}\label{Result_MetSel_ThroughAnyPoint_Repres}\cite[Corollary~3.7]{DFM:MetricIntegral}
		
		Let $F \in \calF$. Through any point $\alpha \in \Graph(F)$ there exists a metric selection which we denote by~${\,s_\alpha}$.
		Moreover,  $F$ has a representation by metric selections, namely
		$$
		F(x) = \{ s_\alpha(x) \ :\ \alpha \in \Graph(F)\}.
		$$
	\end{result}
	\begin{result}\label{Result_LocModuli_s<=LocModuli_v_F}\cite[Theorem~4.9]{BDFM:2021}
		
		Let $F \in \calF$, $s$ be a metric selection of $F$  and $x^*\in[a,b]$. Then
		$$
		\LocalModulCont{s}{x^*}{\delta} \le \LocalModulCont{v_F}{x^*}{ 2\delta}, \quad \delta >0.
		$$
		In particular, if $F$ is continuous at $x^*$, then $s$ is continuous at $x^*$.
	\end{result}
	\begin{result}\label{Result-limit_of_MS_isMS} \cite[Theorem~4.13]{BDFM:2021}
		
		For $F\in \calF$, the pointwise limit (if exists) of a sequence of metric selections of~$F$ is a metric selection~of~$F$.
	\end{result}

	\begin{result}\label{Result_QuasiModulu_MetricSelection}\cite[Lemmas~6.7, 6.8]{BDFM:2021}
		
		Let $F \in \calF$ and $\delta > 0$.  Then any metric selection $s \in \mathcal{S}(F)$ satisfies
		$$
		\NewLeftLocalModul{v_s}{x^*}{\delta} \le \NewLeftLocalModul{v_F}{x^*}{2\delta},  \; x^* \in (a,b], \qquad
		\NewRightLocalModul{v_s}{x^*}{\delta} \le \NewRightLocalModul{v_F}{x^*}{\delta}, \;  x^* \in [a,b).
		$$
	\end{result}
A direct consequence of Lemma~\ref{Lemma_QuasiMod_f<=QuasiMod_v_f} and Result~\ref{Result_QuasiModulu_MetricSelection} is
\begin{corol}\label{corol_QuasiMod_s<=QuasiMod_v_F}
	Let $F \in \calF$, $s \in \mathcal{S}(F)$ and $\delta > 0$.  Then
$$
	\NewLeftLocalModul{s}{x^*}{\delta} \le \NewLeftLocalModul{v_F}{x^*}{2\delta}, \; x^* \in (a,b], \qquad 
\NewRightLocalModul{s}{x^*}{\delta} \le \NewRightLocalModul{v_F}{x^*}{\delta},  \; x^* \in [a,b).
$$
\end{corol}

Now we show that the metric selections of $F$ inherit the Lipschitz regularity property at $x$ from $v_F$.

\begin{lemma}\label{Lip_3}
	Let $F \in \calF$. If $v_F \in \LipAtPoint{x}{ \cal{L} }$ then $s \in \LipAtPoint{x}{4 \cal{L} }$ for all metric selections $s \in \cal{S}(F)$.
\end{lemma}
\begin{proof} 
	Since $v_F \in \LipAtPoint{x}{ \cal{L} }$, there is $\delta>0$ such that ${|v_F(x)- v_F(z)| \leq {\cal L} |x-z|}$, ${\forall z \in \left(x-\frac{\delta}{2}, x+\frac{\delta}{2}\right) \cap [a,b]}$, implying that $\omega(v_F, x, \eta) \le \cal{L} \eta$ for $\eta \le \delta$. Thus by~Result~\ref{Result_LocModuli_s<=LocModuli_v_F} we have for all $z \in \left(x-\frac{\delta}{4}, x+\frac{\delta}{4}\right) \cap [a,b]$
	$$
	|s(x)-s(z)|  \le \omega(s,x,2|x-z|) \le \omega(v_F,x,4|x-z|) \le 4\cal{L} |x-z| .
	$$
\end{proof}

The metric integral of SVFs is introduced in~\cite{DFM:MetricIntegral} and extended to the weighted metric integral in~\cite{BDFM:2021}. We recall its definition.

For a multifunction $\Map{[a,b]}$, a weight function $\kappa:[a,b] \to \R$ and for a partition ${ \chi=\{x_0,\ldots,x_n\} }$, ${a=x_0<x_1<\cdots < x_n=b} $, we define the weighted metric Riemann sum of $F$ by
\begin{align*}
	{\scriptstyle(\cal M_\kappa)} S_{\chi} F & = \left\{ \sum_{i=0}^{n-1} (x_{i+1}-x_i)\kappa(x_i) y_i : \ (y_0,\ldots,y_{n-1}) \in \CH(F(x_0),\ldots,F(x_{n-1})) \right\} \\
	& = \bigoplus_{i=0}^{n-1}(x_{i+1}-x_i)\kappa(x_i) F(x_i).
\end{align*}

The \textbf{weighted metric integral} of $F$ with the weight function $\kappa$ is defined as the upper Kuratowski limit of weighted metric Riemann sums
	$$
	\WeightedMetInt  = \limsup_{|\chi| \to 0} {\scriptstyle(\cal M_\kappa)} S_{\chi} F.
	$$
The set $\WeightedMetInt $ is non-empty whenever the set-valued function $\kappa F$ has a bounded range.

\begin{result}\label{Result_W-MetrInt=W-IntOfMetrSel}~\cite{BDFM:2021}
	Let $F \in \calF$ and  $\kappa \in \BV$. Then the set $\WeightedMetInt$ is compact and
	$$
	\WeightedMetInt=\left \{ \int_{a}^{b}\kappa(x)s(x)dx \ : \ s \in\mathcal{S}(F) \right \}.
	$$
\end{result}
\noindent Here and in such occasions below we understand that the integral is applied to each component of $s = \begin{pmatrix} s_1 \\ \vdots \\ s_d \end{pmatrix}$.


\section {Rate of pointwise convergence of integral operators for real-valued functions}\label{Sect_Convergence_RealValued}
Let $\{\cal K_n(x,t)\}_{n \in \N}$, ${\cal K}_n: [a,b] \times [a,b] \to \R$, be a sequence of functions that are integrable with respect to $t$ for each $x$. 
We term the functions $\cal K_n$ kernels. With the help of the sequence of the kernels we define a sequence of linear integral operators $\{T_n\}_{n \in \N}$  on real-valued functions in $L^\infty[a,b]$ by

\begin{equation} \label{eq:Def_Operator}
	T_nf(x) = \int_a^b \cal K_n(x,t)  f(t)  dt, \quad x \in [a,b], \quad n \in \N.
\end{equation}

\begin{remark}\label{Remark_Integral_Vector_Valued}
If $f: [a,b] \to \Rd$, then $T_nf$ is obtained by the operation of $T_n$ on each component of $f$.
\end{remark}
\medskip

Next we introduce the following notation. For $x \in [a,b]$ let
$$
\alpha_n(x) =  \left| \int_a^b \cal K_n(x,t) dt - 1 \right| 
$$
and
$$
\beta_n(x,\delta) = \int_{|x-t| \ge \delta} |\cal K_n(x,t)| dt , \quad \delta > 0. 
$$
Furthermore, for each $x \in [a,b]$ let $M(x) \in [0,\infty]$ be  such that
\begin{equation} \label{eq:M(x)bound}
	\int_a^b |\cal K_n(x,t)| dt \le M(x), \quad n \in \N.
\end{equation}

\subsection{The case of continuity points}

The theorem below can be considered as a refinement of Theorem 2.1 in \cite[Chapter 1]{DeVoreLorentz:ConstAppr}. 
For that we need the following result~\cite[page~12]{Korovkin:LOandAT}
\begin{lemma}\label{Lemma_Korovkin}
	If the function  $f$ is defined on $[a,b]$ and $I \subseteq [a,b]$ is a closed interval such that $f$ is continuous at all the points of $I$,
	then for any $\epsilon>0$ there exists  $\delta>0$ such that 
	$$
	|f(y)-f(x)| < \epsilon
	$$
	for all $x \in I$, $y \in [a,b]$ and \ $|y-x|<\delta$.
\end{lemma}

\begin{theorem}\label{theorem:DeVoreLorentz}
Let $f:[a,b]\rightarrow \R$ be bounded and measurable 
on $[a,b]$,  $x \in [a,b]$ and let $M(x)$ be as in~\eqref{eq:M(x)bound}. Then\\ 
\noindent {\bf (i)}\;  For all $\delta > 0$ and $n \in \N$
\begin{equation}\label{eq:estim_DeVoreLorentz}
	|T_nf(x) - f(x)| \le \omega(f,x,2\delta) M(x) + \|f\|_{\infty} 2\beta_n(x,\delta) + |f(x)|\alpha_n(x).
\end{equation}

\noindent {\bf (ii)}\; If $\displaystyle \lim_{n \to \infty} \alpha_n(x) = 0$, $\displaystyle \lim_{n \to \infty} \beta_n(x,\delta) = 0$  
for any sufficiently small $\delta > 0$, $M(x) < \infty$, and if $x$ is a point of continuity of~$f$, then
\begin{equation}\label{eq:converg_DeVoreLorentz}
\lim_{n \to \infty} {| T_nf(x) - f(x)|} =0. 
\end{equation}

\noindent {\bf (iii)}\; If $f$ is continuous at all points of a closed interval $I\subseteq [a,b]$, 
if $\displaystyle \lim_{n \to \infty} \alpha_n(x) = 0$ and ${\displaystyle \lim_{n \to \infty} \beta_n(x,\delta) = 0}$ 
uniformly in $x\in I$ for any sufficiently small\, $\delta > 0$, and if $M(x)$ is bounded on $I$, 
then the convergence is uniform in~$I$.
\end{theorem}

\begin{proof}
${}$	\\
\noindent {\bf (i)}\; We have 
\begin{align*}
|T_nf(x) - f(x)| & = \left| \int_a^b \cal K_n(x,t)  f(t)  dt  - \int_a^b \cal K_n(x,t)  f(x)  dt + f(x) \int_a^b \cal K_n(x,t)  dt - f(x) \right| \\
& \le \int_a^b |\cal K_n(x,t)| |f(t) - f(x)| dt + |f(x)| \left| \int_a^b \cal K_n(x,t) dt - 1 \right| \\
& \le \int_{|x-t| \le \delta} |\cal K_n(x,t)| |f(t) - f(x)| dt  + \int_{|x-t| > \delta} |\cal K_n(x,t)| |f(t) - f(x)| dt + |f(x)| \alpha_n(x).
\end{align*} 
Using the fact that $|f(t) - f(x)| \le  \omega(f,x,2\delta) $ in the first integral and that $|f(t) - f(x)| \le 2 \|f\|_\infty$, we obtain 
\begin{align*}
|T_nf(x) - f(x)| & \le  \omega(f,x,2\delta) \int_{|x-t| \le \delta} |\cal K_n(x,t)| dt + 2 \|f\|_\infty \int_{|x-t| > \delta} |\cal K_n(x,t)| dt + |f(x)| \alpha_n(x) \\
& \le \omega(f,x,2\delta) M(x) + 2 \|f\|_\infty \beta_n(x,\delta) + |f(x)| \alpha_n(x) 
\end{align*}
which gives \eqref{eq:estim_DeVoreLorentz}. 

\noindent {\bf (ii)}\; Fix $\delta > 0$ and let $n \to \infty$ in  \eqref{eq:estim_DeVoreLorentz}. With the assumption of~(ii), it follows from~\eqref{eq:estim_DeVoreLorentz} that
$$
\limsup_{n \to \infty} {| T_nf(x) - f(x)|} \le \omega(f,x,2\delta)  M(x).
$$
Since this is valid for each $\delta > 0$ and since $\displaystyle  \lim_{\delta \to 0+}  \omega(f,x,2\delta)= 0$ when $x$ is a point of continuity, we obtain~\eqref{eq:converg_DeVoreLorentz}.

\noindent {\bf (iii)}\; By Lemma~\ref{Lemma_Korovkin}, 
$\omega(f,x,2\delta)$  tends to zero uniformly in $x \in I$ when  $\delta \to 0+$\  and it follows from~\eqref{eq:estim_DeVoreLorentz} that the convergence is uniform in $I$.
\end{proof}

	In view of Lemma~\ref{lemma_IntegralProp}, estimates in the proof of Theorem~\ref{theorem:DeVoreLorentz} can be repeated verbatim for $f: [a,b] \to \Rd$, $f \in \BV$, since such $f$ is Riemann integrable. We obtain

\begin{corol}\label{corol_VectorValued_Rieman}
	If $f: [a,b] \to \Rd$, $f \in \BV$ and $\cal K_n(x,\cdot) \in \BV$ for each $x\in[a,b]$, then the estimates in Theorem~\ref{theorem:DeVoreLorentz} hold.
\end{corol}

\begin{remark}\label{remark_ContEstim_VectorValued}
It is easy to see that if $f: [a,b] \to \Rd$ is bounded and measurable, then~\eqref{eq:estim_DeVoreLorentz} becomes
$$
	|T_nf(x) - f(x)| \le C {\omega}(f,x,2\delta) M(x) + C \|f\|_\infty 2 \beta_n(x,\delta) + C |f(x)| \alpha_n(x), \quad x \in [a,b] , \; n \in \N , \; \delta >0,
$$
where $C$ depends only on the underlying norm $|\cdot|$ in $\Rd$.
\end{remark}

\subsection{The case of discontinuity points}

In this section we follow and refine the analysis in~\cite{Guo:87} for real-valued functions.

Let $f: [a,b] \to \R$,  $f \in \BV$. Fix $x \in (a,b)$, following \cite{Guo:87} we introduce the function $g_x: [a,b] \to \R$,
$$
g_x(t) = \begin{cases}
	f(t) - f(x-), & t \in [a,x),\\
	f(t) - f(x) = 0, & t = x,\\
	f(t) - f(x+), & t \in (x,b].
\end{cases} 
$$
Note that $f \in \BV$ implies that $g_x$ is well defined and is continuous at $x$. It  is easy to check that
$$
f(t) = g_x(t) + f(x-) \chi_{[a,x)}(t) + f(x+) \chi_{(x,b]}(t) + f(x) \chi_{\{x\}}(t), \quad t \in [a,b],
$$
where $\chi_A$ denotes the characteristic function of a set $A$
$$
\chi_A(t) = \begin{cases}
	1, & t \in A,\\
	0, & t \not \in A.
\end{cases}
$$
A simple computation shows that for $t \in [a,b]$
\begin{align*}
	 f(t) & - \frac{1}{2}\big [ f(x+) + f(x-) \big ] = g_x(t) + \frac{1}{2}\big [f(x+) - f(x-)\big ] \big ( \chi_{(x,b]}(t) - \chi_{[a,x)}(t) \big ) \\
	& + 
	\left ( f(x) - \frac{1}{2} \big [ f(x+) + f(x-) \big ] \right ) \chi_{\{x\}}(t) \\
	& = g_x(t) + \frac{1}{2} \big [ f(x+) - f(x-) \big ] \sign(t - x) +  \left ( f(x) - \frac{1}{2} \big [f(x+) + f(x-) \big ] \right ) \chi_{\{x\}}(t),
\end{align*}
where
$$
\sign(t) = \begin{cases}
	1, & t > 0, \\
	0, & t = 0,\\
	-1, & t < 0.
\end{cases}
$$
Inserting this representation into the operator  \eqref{eq:Def_Operator}, we obtain
\begin{equation}\label{eq_1}
T_nf(x) - \frac{1}{2}[f(x+) + f(x-)] \int_a^b \cal K_n(x,t) dt 
= T_n g_x(x) +  \frac{1}{2} [f(x+) - f(x-)] T_n(\sign(\cdot - x))(x).
\end{equation}
Taking into account that $g_x(x) = 0$, the estimate \eqref{eq:estim_DeVoreLorentz} for $T_n g_x(x)$ takes the form
\begin{equation}\label{eq_2}
|T_n g_x (x)| = |T_n g_x(x) - g_x(x)| \le \omega(g_x,x,2\delta) M(x) + 2 \|g_x\|_\infty \beta_n(x,\delta).
\end{equation}

By the definition of the function $g_x$ we easily see that $\|g_x\|_\infty \le 2 \|f\|_\infty$. 
By the definition of the local modulus of continuity and the local quasi-moduli (see Section~\ref{Sect_Prelim_Regularuty}), we get
\begin{equation}\label{eq_3}
\omega(g_x,x,2\delta) \le 2 \max{ \{ \NewLeftLocalModul{f}{x}{\delta}, \NewRightLocalModul{f}{x}{\delta}\} }
= 2 \QuasiLocalModul{f}{x}{\delta}.  
\end{equation}
Clearly, $\frac{1}{2} \big | f(x+) - f(x-) \big | \le \|f\|_\infty$, since $f \in \BV$.
Combining this with ~\eqref{eq_1},~\eqref{eq_2},~\eqref{eq_3} 
we arrive at 
\begin{equation}\label{eq:Long_estim_L_n_Jump}
\begin{split}
	& \left | T_nf(x) - \frac{1}{2}\big [f(x+) + f(x-) \big ] \int_a^b \cal K_n(x,t) dt  \right | \\
	& \le 2\QuasiLocalModul{f}{x}{\delta}  M(x) +
	\|f\|_\infty \left ( 4\beta_n(x,\delta) + \big| T_n(\sign(\cdot - x))(x) \big| \right ). 
\end{split}
\end{equation}
Thus we obtain the following result,

\begin{theorem}\label{theorem_Estim_T_n_Jump}
	Let $f: [a,b] \to \R$,  $f \in \BV$ and $x \in (a,b)$.

\noindent {\bf (i)}\; For all $\delta > 0$ and $n \in \N$ we have
\begin{align*}
	& \left| T_nf(x) - \frac{1}{2}[f(x+) + f(x-)]   \right| \nonumber \\
	& \le 2 \QuasiLocalModul{f}{x}{\delta} M(x) +
	 \|f\|_\infty \left ( 4\beta_n(x,\delta) + \alpha_n(x) + \big| T_n(\sign(\cdot - x))(x) \big| \right ). 
\end{align*}

\noindent {\bf (ii)}\; If $\displaystyle \lim_{n \to \infty} \alpha_n(x) = 0$,  $\displaystyle \lim_{n \to \infty} \beta_n(x,\delta) = 0$ 
for  any sufficiently small $\delta > 0$,  ${\displaystyle \lim_{n \to \infty} T_n(\sign(\cdot - x))(x) = 0 }$, and  $M(x) < \infty$,  then
 $$
\lim_{n \to \infty} { \left| T_nf(x) - \frac{1}{2}[f(x+) + f(x-)]   \right | } =0. 
$$
\end{theorem}

\begin{proof}
	By the triangle inequality and~\eqref{eq:Long_estim_L_n_Jump} we have
\begin{align*}
	& \left| T_n f(x) - \frac{1}{2}[f(x+) + f(x-)] \right| \\
	& \le \left| T_n f(x) - \frac{1}{2}[f(x+) + f(x-)] \int_a^b \cal K_n(x,t) dt  \right|
	+ \frac{1}{2} |f(x+) + f(x-)| \left|\int_a^b \cal K_n(x,t) dt  -1 \right| \\
	& \le  2 \QuasiLocalModul{f}{x}{\delta}  M(x) + \|f\|_\infty \left ( 4\beta_n(x,\delta) + \big| T_n(\sign(\cdot - x))(x)\, \big| \, \right )\,   + \|f\|_{\infty} \alpha_n(x),
\end{align*}
which leads to the first claim. The second claim follows directly from it by applying~\eqref{prop_symm_quasi-modul}.
\end{proof}

Similarly to Corollary~\ref{corol_VectorValued_Rieman} for vector-valued functions we obtain 
	\begin{corol}\label{Corol_Estim_T_n_Jump_VecValued}
		Let $f: [a,b] \to \Rd$, $f \in \BV$, $\cal K_n(z,\cdot) \in \BV$ for each $z\in[a,b]$, then for $x \in (a,b)$ 
		\begin{align*}
			& \left| T_nf(x) - \frac{1}{2}[f(x+) + f(x-)]  \right| 
			\\
			& \le 2 \QuasiLocalModul{f}{x}{\delta} M(x) + 
			\|f\|_\infty \big ( 4\beta_n(x,\delta) + \alpha_n(x)  + \big| T_n(\sign(\cdot - x))(x) \big|\, \big ) .
		\end{align*}

	\end{corol}


\section {Rate of pointwise convergence of integral operators for set-valued functions}\label{Sect_ApproxSVF}

Let $F \in \calF$ and  $\cal K_n(x, \cdot) \in \BV$ for any $x\in [a,b]$. Denoting $\kappa_{n,x}(t) = \cal K_n(x,t)$ and using the concept of the weighted metric integral (see Section~\ref{Sect_MS_MetIntegral}), we define
\begin{equation} \label{eq:Def_Metr_Operator}
T_nF(x) = {\scriptstyle(\cal M_{\kappa_{n,x}})} \int_a^b \kappa_{n,x}(t)  F(t)  dt, \quad x \in [a,b], \quad n \in \N.
\end{equation}
By Result~\ref{Result_W-MetrInt=W-IntOfMetrSel} and since $\kappa_{n,x}(t) = \cal K_n(x,t)$ we have for $F \in \calF$
\begin{equation} \label{eq:Operator_Repr_Metric_Sel}
T_nF (x)  = \left \{ \int_{a}^{b} \cal K_n(x,t) s(t) dt \ : \ s \in\mathcal{S}(F) \right\} = \left \{  T_ns (x) \ : \ s \in\mathcal{S}(F) \right \}, \quad x \in [a,b],
\end{equation}
where $\mathcal{S}(F)$ is the set of metric selections of $F$. By Result~\ref{Result_MetSel_ThroughAnyPoint_Repres} we have as well
\begin{equation} \label{eq:Function_Repr_Metric_Sel}
F (x) = \left \{  s(x) \ : \ s \in\mathcal{S}(F) \right \}, \quad x \in [a,b].
\end{equation}
It is easy to obtain from~\eqref{eq:Operator_Repr_Metric_Sel} and~\eqref{eq:Function_Repr_Metric_Sel} that 
\begin{equation}\label{haus<=sup_s}
\haus \left(T_nF(x), F(x) \right) \le \sup{ \left\{ \left| T_ns(x) - s(x)  \right| \ : \ s \in \mathcal{S}(F) \right\} }. 
\end{equation}
The arguments leading to~\eqref{haus<=sup_s} are similar to those in the proof of Lemma~\ref{lemma_Hausdorff_L^1} below.
\begin{remark} ${}$
 Any bounded linear operator defined for single-valued functions can be extended to set-valued functions from the class $\calF$ by
	$$
	T F (x) = \{ Ts(x)  \ : \ s \in\mathcal{S}(F)  \}, \quad x \in [a,b], \quad F \in \calF.
	$$
	This approach was used for operators $T$ defined on continuous real-valued functions  in \cite[Section~8.2]{DFM:Book_SV-Approx}.
\end{remark}


\subsection{The case of continuity points}\label{Sect_ApproxSVF_Cont}

Recall that by Results~\ref{Result_MetrSel_InheritVariation} and~\ref{Result_LocModuli_s<=LocModuli_v_F} for each selection $s \in \setMS $ we have 
$$
\|s\|_\infty \le \|F\|_\infty \quad \mbox{and} \quad 
\omega(s,x,\delta) \le \omega(v_F, x, 2\delta) \, , \;\, x \in [a,b]\, , \; \delta > 0.
$$
Thus for each $s \in \mathcal{S}(F)$ the estimate in Corollary~\ref{corol_VectorValued_Rieman}  turns into
\begin{equation} \label{Tn(s)-s}
|T_ns(x) - s(x)| \le \omega(v_F,x,4\delta) M(x) +  \|F\|_\infty (2\beta_n(x,\delta) + \alpha_n(x)).
\end{equation}

\noindent We arrive at
\begin{theorem}\label{Theorem_EstimSVF_ContPoint}
Let $F \in \calF$, $x \in [a,b]$, $\{\cal K_n(x,t)\}_{n \in \N}$ be a sequence of kernels on $[a,b]\times[a,b]$ such that $\cal K_n(x, \cdot) \in \BV$ for any $x\in [a,b]$ and let $\alpha_n(x)$, $\beta_n(x,\delta)$ and $M(x)$ be as in~Section~\ref{Sect_Convergence_RealValued}. Then 

	\noindent {\bf (i)}\;  For all $\delta > 0$ and $n \in \N$
	$$
\haus (T_nF(x), F(x) ) \le \omega(v_F,x,4\delta) M(x) + \|F\|_\infty (2\beta_n(x,\delta) + \alpha_n(x)) \, , \quad  x \in [a,b].
	$$

	\noindent {\bf (ii)}\; If $\displaystyle \lim_{n \to \infty} \alpha_n(x) = 0$, $\displaystyle \lim_{n \to \infty} \beta_n(x,\delta) = 0$  for any sufficiently small $\delta > 0$, $M(x) < \infty$ and if $x$ is a point of continuity of~$F$, then
	$$
\lim_{n \to \infty} { \haus \left(T_nF(x), F(x) \right)  } =0. 
	$$

	
	\noindent {\bf (iii)}\; If $F$ is continuous at all points of a closed interval $I \subseteq [a,b]$, if ${\displaystyle \lim_{n \to \infty} \alpha_n(x) = 0}$ 
	and ${\displaystyle \lim_{n \to \infty} \beta_n(x,\delta) = 0}$ uniformly in $x\in I$ for  any sufficiently small $\delta > 0$,  and if $M(x)$ is bounded on $I$, then the convergence is uniform in~$I$.
\end{theorem}
\begin{proof}
The statements~(i) and~(ii) follow from the first two claims of Theorem~\ref{theorem:DeVoreLorentz} combined with~\eqref{haus<=sup_s} and~\eqref{Tn(s)-s}. The proof of the statements~(iii) is based on Result~\ref{Result_f_cont->v_f_cont}, which implies that $v_F$ is uniformly continuous on compact intervals, and therefore $\omega(v_F,x ,4\delta) \to 0$ uniformly for $x \in I$.
\end{proof}

\subsection{The case of discontinuity points}\label{Sect_ApproxSVF_Discont}

For $x \in (a,b)$ we define the set (see also \cite[Section 6.3]{BDFM:2021})
$$
	A_F(x) = \left\{ \frac{1}{2} \left( s(x+) + s(x-) \right) \ : \ s \in \mathcal{S}(F) \right\}.
$$
Similarly to~\eqref{haus<=sup_s} we have
\begin{equation}\label{haus<=sup_s_Jump}
\haus \left(T_nF(x), A_F(x) \right) \le \sup{ \left\{ \left| T_ns(x) - \frac{1}{2}(s(x+) + s(x-))   \right| \ : \ s \in \mathcal{S}(F) \right\} }. 
\end{equation}
 Recall that by Result~\ref{Result_MetrSel_InheritVariation}, $\|s\|_\infty \le \|F\|_\infty$\, and $\frac{1}{2}|s(x+) + s(x-)| \le \|F\|_\infty$. 
In view of Corollary~\ref{corol_QuasiMod_s<=QuasiMod_v_F} we have for each $s \in \mathcal{S}(F)$ 
$$
\QuasiLocalModul{s}{x}{\delta} \le \QuasiLocalModul{v_F}{x}{2\delta}.
$$
By Corollary~\ref{Corol_Estim_T_n_Jump_VecValued}
and the above inequality we have for any $s \in \mathcal{S}(F)$ 
\begin{align*}\label{eq:Long_estim_L_n_Jump}
	 & \left| T_n s(x)   - \frac{1}{2}[s(x+) + s(x-)] \right| \\
	 & \le  2  \QuasiLocalModul{v_F}{x}{2\delta} M(x) + \|F\|_\infty \big ( 4\beta_n(x,\delta) + \alpha_n(x) + \big| T_n(\sign(\cdot - x))(x)\, \big| \, \big ). 
\end{align*}

This together with~\eqref{haus<=sup_s_Jump}, by arguments  as in the proof of~(ii) of Theorem~\ref{theorem:DeVoreLorentz}, leads to the following result

\begin{theorem}\label{Theorem_EstimSVF_Jamp}
Let $F \in \calF$ and let $\{\cal K_n(x,t)\}_{n \in \N}$ be a sequence of kernels on $[a,b]\times[a,b]$ such that $\cal K_n(x, \cdot) \in \BV$ for any $x\in [a,b]$. 
Then for any $x\in (a,b)$ the following statements hold.
\medskip

\noindent {\bf (i)}\;	For all $\delta > 0$ and $n \in \N$
\begin{align*}
	\haus (T_nF(x), A_F(x) ) \le  2 \QuasiLocalModul{v_F}{x}{2\delta} M(x) + 
	\|F\|_\infty \big ( 4 \beta_n(x,\delta) +\alpha_n(x)  +  \big| T_n(\sign(\cdot - x))(x) \big| \big ).
\end{align*}

\noindent {\bf (ii)}\;	
	If $\displaystyle \lim_{n \to \infty} \alpha_n(x) = 0$, 
	$\displaystyle \lim_{n \to \infty} \beta_n(x,\delta) = 0$  for any sufficiently small $\delta > 0$,
  ${\displaystyle \lim_{n \to \infty} T_n(\sign(\cdot - x))(x) = 0}$, and $M(x) < \infty$,  then 
	$$
	\lim_{n \to \infty} { \haus \left(T_nF(x), A_F(x) \right)  } =0. 
	$$

\end{theorem}


\section{Specific operators}\label{Sect_SpecificOperators}
Here we apply the results of the previous sections to two specific integral operators. 
In particular, at points of discontinuity we obtain error estimates that combine ideas from~\cite{Guo:87, ZengPiriou:98} with the local quasi-moduli of continuity~\eqref{Left_QuasiModuli}, \eqref{Right_QuasiModuli}.

\subsection{Bernstein-Durrmeyer operators}

\noindent For $x \in [0,1]$, $n \in \N$ the Bernstein basis polynomials are defined as
\begin{equation}\label{Bernstein-basis}
p_{n,k}(x) = {n \choose k} x^k (1-x)^{n-k} \, ,  \quad k = 0, 1, \ldots, n.
\end{equation}
Note that $\displaystyle \sum_{k=0}^n p_{n,k}(x) = 1$ and $\displaystyle \int_0^1 p_{n,k}(x) dx = \frac{1}{n+1}$. 

\noindent The Bernstein-Durrmeyer  operator is defined for $f \in L^1[0,1]$ by
$$
M_nf(x) = (n+1) \sum_{k=0}^n p_{n,k}(x) \int_0^1 p_{n,k}(t) f(t) dt = \int_0^1 \cal K_n(x,t) f(t) dt\, , \quad x \in [0,1]\, , \;  n \in \N,
$$
where
\begin{equation}\label{eq:Kernel_BDO}
\cal K_n(x,t) = (n+1) \sum_{k=0}^n p_{n,k}(x) p_{n,k}(t).
\end{equation}
The Bernstein-Durrmeyer operator for a set-valued function $F \in \cal F[0,1] $ is
$$
M_nF(x) = {\scriptstyle(\cal M_{\kappa_{n,x}})} \int_0^1 \kappa_{n,x}(t) F(t)  dt, \quad x \in [0,1], \quad n \in \N\, ,
$$
with $\kappa_{n,x}(t)=\cal K_n(x,t)$, where $\cal K_n(x,t)$ is given in~\eqref{eq:Kernel_BDO}.

\noindent The properties of the Bernstein basis polynomials $p_{n,k}$ yield 
$$
\cal K_n(x,t) \ge 0, \quad \int_0^1 \cal K_n(x,t) dt = 1,
$$
so that $\alpha_n(x) = 0$ and $M(x) = 1$, $x \in [0,1]$. 

Denote $e_i(t) = t^i$, $i=0,1,2$. A direct calculation shows that (see e.g.~\cite{Guo:87})
$$
M_n e_0 (x) = 1 \, , \qquad M_n e_1(x) = \frac{nx+1}{n+2}\, , \qquad M_n e_2(x) = \frac{n(n-1) x^2 + 4nx +2}{(n+2)(n+3)} . 
$$
This implies
\begin{equation}\label{eq:BDO_sec_moment}
M_n ((\cdot - x)^2)(x) = \frac{2[ (n-3)x(1-x) + 1]}{(n+2)(n+3)}.
\end{equation}
Following S. Guo (see Lemma~6 in  \cite{Guo:87}), we estimate $\beta_n(x,\delta)$ as follows:  
\begin{align*}
\beta_n(x,\delta) &= \int_{|x-t| \ge \delta} \cal K_n(x,t) dt 
\le \int_{|x-t| \ge \delta} \left( \frac{x-t}{\delta} \right)^2 \cal K_n(x,t) dt \\
& \le \frac{1}{\delta^2} \int_0^1 (x-t)^2 \cal K_n(x,t) dt 
=  \frac{1}{\delta^2} M_n ((\cdot - x)^2)(x) 
= \frac{1}{\delta^2} \frac{2[ (n-3)x(1-x) + 1]}{(n+2)(n+3)}.
\end{align*}
Using the inequality $ \displaystyle x(1-x) \le \frac{1}{4}$, $x \in [0,1]$, we arrive at the estimate
$\displaystyle
\beta_n(x, \delta) \le \frac{n+1}{2 \delta^2 (n+2)(n+3)} \le \frac{1}{2n\delta^2 }.
$
Thus, for the Bernstein-Durrmeyer operator we have
$$
M(x) = 1 \; , \; \alpha_n(x) = 0 \; , \;  \beta_n(x, \delta) \le \frac{1}{2n\delta^2 } \, , \quad x \in [0,1].
$$


\noindent \textbf{The case of continuity points} 
\smallskip

The assumptions of Theorem~\ref{theorem:DeVoreLorentz}  are fulfilled for the Bernstein-Durrmeyer operator.
  Thus we obtain the following result, where part {\bf(i)}, which provides rate of convergence, is new in this form, while parts {\bf (ii)} and {\bf (iii)} are already known.

\begin{corol}\label{Lemma_estim_BDO_RV}
	Let $f \in L^1[0,1]$ be bounded on $[0,1]$,  $x \in [0,1]$. Then
	\medskip
		
	\noindent {\bf (i)}\;  For all $\delta > 0$ and $n \in \N$
$$
|M_nf(x) - f(x)| \le \omega(f,x,2\delta)  + \|f\|_\infty \frac{1}{n\delta^2 }.
$$
	\noindent {\bf (ii)}\; If $x$ is a point of continuity of~$f$, then $ \lim_{n \to \infty} {| M_nf(x) - f(x)|} = 0 $.
	\medskip

%
	\noindent {\bf (iii)}\; If $f$ is continuous at all points of a closed interval $I\subseteq[0,1]$, 
	then the convergence is uniform in~$I$.
\end{corol}

From Theorem~\ref{Theorem_EstimSVF_ContPoint} for the Bernstein-Durrmeyer operator we get
\begin{corol}\label{Lemma_estim_BDO_SV}
	Let $F \in \mathcal{F}[0,1]$, $x \in [0,1]$. Then 
	\medskip  
		
	\noindent {\bf (i)}\;  For all $\delta > 0$ and $n \in \N$
$$
\haus (M_nF(x), F(x) ) \le \omega(v_F,x,4\delta) +  \|F\|_\infty \frac{1}{n\delta^2 }.
$$
	
	\noindent {\bf (ii)}\; If $x$ is a point of continuity of~$F$, then $\lim_{n \to \infty} { \haus \left(M_nF(x), F(x) \right)  } =0 $.
	\medskip  
	
%
	\noindent {\bf (iii)}\; If $F$ is continuous at all points of a closed interval $I\subseteq[0,1]$, then the convergence is uniform in~$I$.
\end{corol}
\medskip


\noindent \textbf{The case of discontinuity points} 
\smallskip

Theorem~\ref{theorem_Estim_T_n_Jump} for the Bernstein-Durrmeyer operator leads to
\begin{align*}
	\left| M_nf(x) - \frac{1}{2}[f(x+) + f(x-)]  \right| 
	 \le 2 \QuasiLocalModul{f}{x}{\delta} 
	+ \|f\|_\infty \left ( \frac{2}{n\delta^2} + \big | M_n(\sign(\cdot - x))(x) \big | \right ) 
\end{align*}
for $x \in (0,1)$. S.~Guo proved in \cite{Guo:87} (see Proof of the Theorem) that 
$$
| M_n(\sign(\cdot - x))(x) | \le \frac{13}{2 \sqrt{nx(1-x)}}
$$
for $x \in (0,1)$ and $n$ large enough. Thus we get

\begin{corol}\label{Corol_Estim_BDO_RV_Jamp}
	Let $f: [0,1] \to \R$, $f \in \BV$ and $x \in (0,1)$. Then

\noindent {\bf (i)}\; 
For all $\delta > 0$ and $n \in \N$  large enough 
$$
	\left| M_nf(x) - \frac{1}{2}[f(x+) + f(x-)]  \right|  \le 2 \QuasiLocalModul{f}{x}{\delta}
	+ \|f\|_\infty \left ( \frac{2}{n\delta^2} + \frac{13}{2 \sqrt{nx(1-x)}} \right ),
$$
\noindent {\bf (ii)}\;  $\displaystyle	\lim_{n \to \infty} { \left| M_nf(x) - \frac{1}{2}[f(x+) + f(x-)]   \right | } =0 $.
\end{corol}
Note that (ii) and an estimate for the order of convergence are obtained by S.~Guo in~\cite{Guo:87}.

\noindent Theorem~\ref{Theorem_EstimSVF_Jamp} for the Bernstein-Durrmeyer operator results in
\begin{corol}\label{Corol_Estim_BDO_SVF_Jamp}
	Let $F \in \mathcal{F}[0,1]$, $x \in (0,1)$. Then 
	
\noindent {\bf (i)}\; For all $\delta > 0$ and large enough $n \in \N$ 
$$
		\haus (M_nF(x), A_F(x) ) \le  2 \QuasiLocalModul{v_F}{x}{2\delta}+  \|F\|_\infty \left ( \frac{2}{n\delta^2} + \frac{13}{2 \sqrt{nx(1-x)}} \right ).
$$

\noindent {\bf (ii)}\;  $\displaystyle \lim_{n \to \infty} { \haus \left(M_nF(x), A_F(x) \right)  } =0$.
\end{corol}


\subsection{Kantorovich operators}  

For $f \in L^1[0,1]$ the Kantorovich operator is defined by
$$
K_nf(x) = (n+1) \sum_{k=0}^n p_{n,k}(x) \int_{\frac{k}{n+1}}^{\frac{k+1}{n+1}} f(t) dt = \int_0^1 \cal K_n(x,t) f(t) dt, \quad x \in [0,1], \quad n \in \N,
$$
where
$$
\cal K_n(x,t) = (n+1) \sum_{k=0}^n p_{n,k}(x) \chi_{\left[  \frac{k}{n+1}, \frac{k+1}{n+1} \right]} (t)
$$
and $p_{n,k}(x)$ are defined in~\eqref{Bernstein-basis}.

By formula \eqref{eq:Def_Metr_Operator} we can extend this operator to set-valued functions $F \in \mathcal{F}[0,1]$. 
Using the properties of $p_{n,k}(x)$, we obtain, as in the case of the Bernstein-Durrmeyer operator, that $\cal K_n(x,t) \ge 0$,
${\int_0^1 \cal K_n(x,t) dt = 1 }$, leading to $\alpha_n(x) = 0$ and $M(x) = 1$, $x \in [0,1]$. 


Calculating $K_n e_i(x)$, $i=0,1,2$ one can get (see e.g. \cite{ZengPiriou:98})
$$
K_n e_0(x) = 1, \qquad  K_n e_1(x) = \frac{2nx+1}{2(n+1)}, \qquad K_n e_2(x) = \frac{3n(n-1) x^2 + 6nx + 1}{3(n+1)^2},
$$
and
\begin{equation}\label{eq:Kant_sec_moment}
K_n ((\cdot - x)^2)(x) = \frac{3 (n-1)x(1-x) + 1}{3(n+1)^2}.
\end{equation}
Thus, estimating $\beta_n(x,\delta)$ as in the case of the Bernstein-Durrmeyer operator, we obtain
$$
\beta_n(x,\delta) \le   \frac{1}{\delta^2} K_n ((\cdot - x)^2)(x) =  \frac{1}{\delta^2} \frac{3 (n-1)x(1-x) + 1}{3(n+1)^2} \le \frac{1}{4n\delta^2 } .
$$

\noindent \textbf{The case of continuity points} 
\smallskip

Theorem~\ref{theorem:DeVoreLorentz} takes for the Kantorovich operator the following form, 
 where part {\bf(i)} is new in this form (providing rate of convergence), while parts {\bf (ii)} and {\bf (iii)} are well-known.
	

\begin{corol}\label{Lemma_estim_Kant_RV}
	Let $f \in L^1[0,1]$ be bounded on $[0,1]$,  $x \in [0,1]$. Then 
	\medskip
	
	\noindent {\bf (i)}\;  For all $\delta > 0$ and $n \in \N$
	$$
	|K_nf(x) - f(x)| \le \omega(f,x,2\delta)  + \|f\|_\infty \frac{1}{2n\delta^2 }.
	$$
	
	\noindent {\bf (ii)}\; If $x$ is a point of continuity of~$f$, then $ \lim_{n \to \infty} {| K_nf(x) - f(x)|} = 0 $.
	\medskip
	
%
	\noindent {\bf (iii)}\; If $f$ is continuous at all points of a closed interval $I\subseteq[0,1]$, then the convergence is uniform in~$I$.
\end{corol}

Applying Theorem~\ref{Theorem_EstimSVF_ContPoint}, we get

\begin{corol}\label{Lemma_estim_BDO_SV}
	Let $F \in \mathcal{F}[0,1]$, $x \in [0,1]$. Then 
	\medskip  
	
	\noindent {\bf (i)}\;  For all $\delta > 0$ and $n \in \N$
	$$
	\haus (K_nF(x), F(x) ) \le \omega(v_F,x,4\delta) +  \|F\|_\infty \frac{1}{2n\delta^2 }.
	$$
	
	\noindent {\bf (ii)}\; If $x$ is a point of continuity of~$F$, then $\lim_{n \to \infty} { \haus \left(K_nF(x), F(x) \right)  } =0 $.
	\medskip  
	
%
	\noindent {\bf (iii)}\; If $F$ is continuous at all points of a closed interval $I\subseteq[0,1]$, then the convergence is uniform in~$I$.
\end{corol}
\smallskip

\noindent \textbf{The case of discontinuity points} 
\smallskip

Theorem~\ref{theorem_Estim_T_n_Jump} for the Kantorovich operator gives the estimate
\begin{align*}
	\left| K_nf(x) - \frac{1}{2}[f(x+) + f(x-)]  \right| 
	\le 2 \QuasiLocalModul{f}{x}{\delta}  
	+ \|f\|_\infty \left ( \frac{1}{n\delta^2} + \big | K_n(\sign(\cdot - x))(x) \big | \right ) 
\end{align*}
for $x \in (0,1)$. An estimate for $\big | K_n(\sign(\cdot - x))(x) \big |$ is given in~\cite{ZengPiriou:98}. Closely following the consideration of Zeng and Piriou in~\cite{ZengPiriou:98}, we derive here a very similar estimate in a slightly different form.
We also replace the estimate $p_{n,k}(x) \le \frac{1}{\sqrt{2e}} \frac{1}{\sqrt{nx(1-x)}}$ used in~\cite{ZengPiriou:98} and quoted there from an unpublished paper,  by the estimate
\begin{equation}\label{EstimateGuoBasis}
	p_{n,k}(x) \le \frac{5}{2\sqrt{nx(1-x)}}
\end{equation} 
of Guo, which is published in~\cite{Guo:87} with a full proof. 

Let $x \in \left[ \frac{\ell}{n+1}, \frac{\ell + 1}{n+1} \right)$ with a certain $\ell \in \{0,1,\ldots,n\}$, then
\begin{align*}
	K_n(\sign(\cdot - x))(x) & 
	= (n+1) \sum_{k=0}^n p_{n,k}(x) \int _{\frac{k}{n+1}}^{\frac{k+1}{n+1}} \sign(t-x) dt \\
	& = - \sum_{k=0}^{\ell -1} p_{n,k}(x) + \sum_{k=\ell+1}^n p_{n,k}(x) + p_{n,\ell}(x) (n+1) \left( \frac{\ell+1}{n+1} - x - \left(x - \frac{\ell}{n+1} \right)  \right) \\
	& = 2 \sum_{k=\ell+1}^n p_{n,k}(x) - 1 + 2 p_{n,\ell}(x) [(\ell +1) - (n+1)x],
\end{align*}
and, since $\ell \le (n+1)x  < \ell + 1$,
$$
\big | K_n(\sign(\cdot - x))(x) \big | \le 2 \left| \sum_{(n+1)x < k \le n} p_{n,k}(x) - \frac{1}{2} \right| + 2 p_{n,\ell}(x).
$$
It was proved in~\cite{ZengPiriou:98} (see proof of Lemma~2) that
$$
\left| \sum_{nx < k \le n} p_{n,k}(x) - \frac{1}{2} \right| \le  \frac{0.8(2x^2-2x+1)}{\sqrt{nx(1-x)}} < \frac{1}{\sqrt{nx(1-x)}},
$$
where the last inequality is easy to check.
Combining this with \eqref{EstimateGuoBasis}, we obtain 
$$ 
\left| \sum_{(n+1)x < k \le n} p_{n,k}(x) - \frac{1}{2} \right| \le \frac{7}{2\sqrt{nx(1-x)}} 
$$ 
and finally
$$
\big | K_n(\sign(\cdot - x))(x) \big | \le  \frac{12}{\sqrt{nx(1-x)}}
$$
for $x \in (0,1)$. Thus we obtain

\begin{corol}\label{Corol_Estim_KO_RV_Jamp}
	Let $f: [0,1] \to \R$, $f \in \BV$ and $x \in (0,1)$. Then
	
	\noindent {\bf (i)}\; 
	For all $\delta > 0$ and $n \in \N$ 
	$$
	\left| K_nf(x) - \frac{1}{2}[f(x+) + f(x-)]  \right|  \le 2  \QuasiLocalModul{f}{x}{\delta}
	+ \|f\|_\infty \left ( \frac{1}{n\delta^2} + \frac{12}{ \sqrt{nx(1-x)}} \right ),
	$$
	\noindent {\bf (ii)}\;  $\displaystyle	\lim_{n \to \infty} { \left| K_nf(x) - \frac{1}{2}[f(x+) + f(x-)]   \right | } =0 $.
\end{corol}

Note that {\bf (ii)} is known (see~\cite{ZengPiriou:98} and the bibliography therein), this is also the case for the order of convergence $1/\sqrt{n}$. 

Theorem~\ref{Theorem_EstimSVF_Jamp} for the Kantorovich operator takes the following form.
\begin{corol}\label{Corol_Estim_KO_SVF_Jamp}
	Let $F \in \mathcal{F}[0,1]$, $x \in (0,1)$. Then 
	
	\noindent {\bf (i)}\; For all $\delta > 0$ and  $n \in \N$ 
	$$
	\haus (K_nF(x), A_F(x) ) \le 2 \QuasiLocalModul{v_F}{x}{2\delta} +  \|F\|_\infty \left ( \frac{1}{n\delta^2} + \frac{12}{ \sqrt{nx(1-x)}} \right ).
	$$

	\noindent {\bf (ii)}\;  $\displaystyle \lim_{n \to \infty} { \haus \left(K_nF(x), A_F(x) \right)  } =0$.
\end{corol}

\subsection{Rate of convergence of the Bernstein-Durrmeyer operators and the Kantorovich operators for locally Lipschitz functions} \label{Section_Rate_Pointwise}
\smallskip

In this section the sequence $\{T_n\}$ is the sequence of Bernstein-Durrmeyer operators $\{M_n\}$ or the sequence of Kantorovich operators $\{K_n\}$. 

Using the notion of locally Lipschitz functions we can state the following result.

\begin{theorem}\label{Theorem_ConvergenceRate_BDO}
	Let $F \in \mathcal{F}[0,1]$ be locally Lipschitz around a point $x \in [0,1]$. Then 
	$$
	\haus \left(T_nF(x), F(x) \right) = O\left( \frac{1}{\sqrt{n}}\right),  \quad n \to \infty.
	$$
\end{theorem}

\begin{proof}
	By~\eqref{haus<=sup_s}
	$$
	\haus \left(T_nF(x), F(x) \right) \le \sup{ \left\{ \left| T_ns(x) - s(x)  \right| \ : \ s \in \mathcal{S}(F) \right\} }.
	$$
	Since $F$ is Lipschitz around the point $x$ with some $\mathcal{L} > 0$, by~(ii) of Lemma~\ref{Lip_1}, $v_F \in \LipAtPoint{x}{ \cal{L} }$
	and by Lemma~\ref{Lip_3} each metric selection $s \in \cal{S}(F)$ satisfies $s \in \LipAtPoint{x}{4 \cal{L} }$. By Lemma~\ref{lemma:loc_Lip_glob_Lip_at_point} 
	there exists $\widetilde{\mathcal{L}} > 0$ such that 

		\begin{equation}\label{number}
			|s(z) - s(x)| \le \widetilde{\mathcal{L}} |z - x|, \quad \forall z \in [0,1].
		\end{equation}
		Using Lemma~\ref{lemma_IntegralProp}, \eqref{number}, the Cauchy-Schwarz inequality and \eqref{eq:BDO_sec_moment} for $\{M_n\}$ or \eqref{eq:Kant_sec_moment} for $\{K_n\}$, we get 
	\begin{align*}
		\left|T_ns(x) - s(x)  \right| & \le \int_0^1 |s(t) - s(x)| \cal K_n(x,t) dt \le \widetilde{\mathcal{L}} \int_0^1 |t - x| \cal K_n(x,t) dt \\
		& \le \widetilde{\mathcal{L}} \left( \int_0^1 |t - x|^2 \cal K_n(x,t) dt \right)^{\frac{1}{2}}   \left( \int_0^1  \cal K_n(x,t) dt \right)^{\frac{1}{2}} = O\left( \frac{1}{\sqrt{n}}\right),  \quad n \to \infty,
	\end{align*}
	and the statement follows.
\end{proof}


\section{Approximation of the set of metric selections in $\bold {L^1[a,b]}$}\label{Sect_TwoSets}

In this section we consider the sequence of operators $\{T_n\}$ defined in~\eqref{eq:Def_Operator} and study two sets of functions, $\setMS$ and the set
$$
{T_n\setMS } = \{ T_n s : \,  s\in \setMS \} .
$$
Note that for $F \in \calF$, the set $\setMS \subset L^1[a,b]$, since it consists of functions of bounded variation.
We assume that $\cal K_n(\cdot, \cdot) \in C([a,b]^2) $. This condition  guaranties 
that $T_n: \setMS \to C[a,b]\subset L^1[a,b]$.

Motivated by~\eqref{eq:Operator_Repr_Metric_Sel}, we regard the set of functions $T_n\setMS$ as an approximant to $F$, represented, in view of Result~\ref{Result_MetSel_ThroughAnyPoint_Repres}, by the set of functions $\setMS$.
We show below that $\setMS$ and $T_n\setMS$ are elements of the metric space $\widetilde{H}$
of compact non-empty subsets of $L^1[a,b]$, endowed with the Hausdorff metric. We use this metric to measure the approximation error of $\setMS$ by $T_n\setMS$.
\smallskip

\subsection{Two compact sets of functions}

In the the next two lemmas we prove that the sets  $\setMS$, $T_n\setMS$ are compact  in $L^1{[a,b]}$, 
thus they are elements of $ \widetilde{H} $. 
It is enough to show that they are sequentially compact, i.e. that every sequence in  $\setMS$ (in $T_n\setMS$, respectively) has a convergent subsequence with a limit in $\setMS$ (in $T_n\setMS$,  respectively).

	\begin{lemma}\label{SF_compact}
		For $F \in \calF$ the set $\setMS$ is compact in $L^1{[a,b]}$.
	\end{lemma}
	\begin{proof}
		By Result~\ref{Result_MetrSel_InheritVariation} any metric selection $s\in \setMS$ satisfies $\|s\|_\infty\le \|F\|_\infty$  and $V_a^b(s) \le V_a^b(F)$. Applying Helly's selection principle, we conclude that for any sequence of metric selections there is a subsequence converging pointwisely at all points $x \in [a,b]$. By Result~\ref{Result-limit_of_MS_isMS} the pointwise limit function of such a subsequence is a metric selection. Thus by Lebesgue Dominated Convergence Theorem this subsequence converges in the $L^1$-norm to the same limit metric selection.
	\end{proof}
	
 	\begin{lemma}\label{T_nSF_compact}
		Let $F \in \calF$ and assume that $\cal K_n(\cdot, \cdot) \in C([a,b]^2) $. 
		Then the set $T_n\setMS$ is compact in $L^1[a,b]$.
	\end{lemma}
	
	\begin{proof} 
		We have to show that any sequence $\{\sigma_k\}_{k=1}^\infty \subset T_n\setMS$ has a subsequence converging  in the $L^1$-norm  and that its $L^1$-limit is in $T_n\setMS$.
		
		By definition 
		$$
		\sigma_k (x) = T_n s_k(x)  = \int_{a}^{b} \cal K_n(x ,t) s_k(t) dt, 
		$$ 
		for some $s_k \in \setMS$.
		
		By Helly's selection principle there exists a subsequence $\{s_{k_\ell} \}_{\ell=1}^\infty$ such that $\displaystyle \lim_{\ell\to \infty} s_{k_\ell}(t)=s^\infty (t)$ pointwisely for all $t \in [a,b]$. By Result~\ref{Result-limit_of_MS_isMS}, $s^\infty \in \setMS$. Since $\cal K_n$ is bounded, we have  $\displaystyle \lim_{\ell\to \infty} \cal K_n(x ,t) s_{k_\ell}(t) = \cal K_n(x ,t)s^\infty (t)$,  for all $x, t \in [a,b]$. Clearly, the functions $\cal K_n(x ,t) s_{k_\ell}(t)$, $\ell \in \N$ 
		are dominated by $\|\cal K_n\|_\infty \|F\|_\infty$. Applying Lebesgue Dominated Convergence Theorem we get
		\begin{equation}\label{TnMS_def}
			\lim_{\ell \to \infty} \sigma_{k_\ell}(x) = \lim_{\ell \to \infty} \int_{a}^{b} \cal K_n(x,t) s_{k_\ell}(t)dt = \int_{a}^{b} \cal K_n(x,t) s^\infty(t)dt  =\sigma^\infty(x), \qquad x\in [a,b].
		\end{equation}
		Thus, we have that $\{\sigma_{k_\ell}(x)\}_{\ell=1}^\infty$ converges pointwisely to $\sigma^\infty(x)$ for all $x\in [a,b]$. Also the sequence $\{\sigma_{k_\ell}\}_{\ell=1}^\infty$ is dominated by $\|F\|_\infty \|\cal K_n\|_\infty (b-a)$, indeed 
		$$
		| \sigma_{k_\ell}(x) | \le \int_{a}^{b} | \cal K_n(x ,t)| \, |s_{k_\ell}(t)| dt \le 
		\|F\|_\infty \int_{a}^{b} | \cal K_n(x ,t)| dt \le \|F\|_\infty \|\cal K_n\|_\infty (b-a ) , \quad x \in [a,b].
		$$
		It follows that the sequence of functions $ \{ \sigma_{k_\ell}(x) \}$,  ${\ell \in \N} $, $x \in [a,b]$ satisfies the assumptions of the Lebesgue Dominated Convergence Theorem and we get 
		$$
		\lim_{\ell \to \infty } \int_{a}^{b}| \sigma_{k_\ell}(x) - \sigma^\infty(x) | dx = 0.
		$$
		To finish the proof, note that by~\eqref{TnMS_def} $\sigma^\infty \in T_n\setMS$, since $s^\infty  \in \setMS$.
	\end{proof}

\noindent The Hausdorff distance between the sets $\setMS$ and $T_n\setMS$ is 
$$
	\haus(\setMS , T_n\setMS ) =\max \left \{ \sup_{ \sigma \in T_n\setMS} \, \inf_{ s\in \setMS } \| s-\sigma\|_{L^1} \, , \, 
	\sup_{s \in \setMS }  \inf_{\sigma \in T_n\setMS }\| s-\sigma\|_{L^1} \right \} .
$$

\begin{lemma}\label{lemma_Hausdorff_L^1}
Let  $F \in \calF$. For the sets $\setMS$ and $T_n\setMS$
$$
\haus(\setMS , T_n\setMS) \le \sup_{s\in \setMS} \| s-T_n s\|_{L^1} .
$$
\end{lemma}
\begin{proof}
For any fixed $s^* \in \setMS$ we have 
$$
\displaystyle \inf_{\sigma \in T_n\setMS }\| s^* -\sigma \|_{L^1}\le \|s^* - T_ns^*\|_{L^1} \le \sup_{s\in \setMS} \| s-T_n s\|_{L^1}.
$$
Similarly, for any fixed  $\sigma^* \in T_n\setMS$ there is $s^* \in \setMS$ 
with $T_n s^* = \sigma^* $ and we get
$$
\displaystyle \inf_{s\in \setMS } \| s - \sigma^* \|_{L^1}\le\| s^*-T_ns^* \|_{L^1} \le \sup_{s\in \setMS} \| s-T_n s\|_{L^1}. 
$$ 
Now the statement follows easily.
\end{proof}

\subsection{Error estimates}\label{Sect_Error bounds_in_$L^1$}

For $e_i(x)=x^i$, denote ${\displaystyle \lambda_{n} = ( \max_{i=0,1} \| T_n e_i  - e_i\|_{L^1})^{1/2}}$.
Here we use the integral moduli $\vartheta$ and $\vartheta_{2}$ defined in Section~\ref{Sect_Prelim_Regularuty}.
In light of the Theorem in~\cite{Berens_DeVore} we get for real-valued functions  (as a special case of the statement in~\cite{Berens_DeVore} for $p=1$)

\begin{result}\label{Result_Rate_L1}
	Let $T_n$ be a positive linear operator satisfying  $\|T_ng  \|_{L^1}  \le \|g  \|_{L^1} $ for any $g \in L^1[a,b]$. Then 
	\begin{equation}\label{Result_1_Berens-DeVore}
	\|T_n f - f \|_{L^1} \le C_1 \left (\lambda_{n}^{2} \| f \|_{L^1}+  \vartheta_{2}(f, \lambda_{n} ) \right ), \quad f \in L^1[a,b] .
	\end{equation}
	If, in addition $T_n e_0 \equiv e_0$, then 
	\begin{equation}\label{Result_2_Berens-DeVore}
		\|T_n f - f \|_{L^1} \le C_2 \left (\lambda_{n}^{2} \, \ModulContPOne{f}{b-a} +  \vartheta_{2}\big ( f, \lambda_{n} \big ) \right ) , \quad f \in L^1[a,b].
	\end{equation}
	Here the constants $C_i$, $i=0,1$ depend only on the interval $[a,b]$.
\end{result}
Applying~\eqref{Result_1_Berens-DeVore}  and~\eqref{Result_2_Berens-DeVore} on a metric selection $s \in \cal{S}(F)$ and using~\eqref{relation_moduli} and~\eqref{estimate_modulu_p=1} we get
$$
\|T_n s - s \|_{L^1} \le \widetilde{C}_1 \left ( \lambda_{n}^{2} \, \| s \|_{L^1} + 2 \ModulContPOne{s}{\lambda_{n} }\right ) \le  
\widetilde{C}_1 \left ( \lambda_{n}^{2} \, \| s \|_{L^1}   + 2 \lambda_{n} V_a^b(s) \right ) ,
$$
and in the case when $T_n e_0 \equiv e_0$
$$
	\|T_n s - s \|_{L^1} \le \widetilde{C}_2 \left ( \lambda_{n}^{2} \, \ModulContPOne{s}{b-a}+ 2 \ModulContPOne{s}{\lambda_{n} }\right ) \le  
	\widetilde{C}_2 \left ( \lambda_{n}^{2} \, (b-a) V_a^b(s)  + 2 \lambda_{n} V_a^b(s) \right ) ,
$$
where $\widetilde{C}_i$, $i=0,1$ depend only on $[a,b]$ and on the underlying norm in $\Rd$.

\noindent Since by Result~\ref{Result_MetrSel_InheritVariation}  $V_a^b(s) \le V_a^b(F)$, we obtain for any $s \in \setMS$ 
\begin{equation}\label{Concequence_From_Berens-DeVore_1}
	\sup_{s\in \setMS} \|T_n s - s \|_{L^1}  \le  \widetilde{C}_1 \left ( \lambda_{n}^{2} \, \sup_{s\in \setMS} \| s \|_{L^1}   + 2 \lambda_{n} V_a^b(F) \right ) ,
\end{equation}
and if  $T_n e_0 \equiv e_0$
\begin{equation}\label{Concequence_From_Berens-DeVore_2}
	\sup_{s\in \setMS} \|T_n s - s \|_{L^1}  \le  \widetilde{C}_2 \left ( \lambda_{n}^{2} \, (b-a)  + 2 \lambda_{n}  \right ) V_a^b(F) .
\end{equation}
Thus in view of~Lemma~\ref{lemma_Hausdorff_L^1}, \eqref{Concequence_From_Berens-DeVore_1} and~\eqref{Concequence_From_Berens-DeVore_2} we arrive at

\begin{theorem}\label{Theorem_ErrorBounds_L1}
		Let $T_n$ be a positive linear operator satisfying  $\|T_ng  \|_{L^1}  \le \|g  \|_{L^1} $ for any $g \in L^1[a,b]$ and let $F \in \calF$. Then 
$$
	\haus(\setMS , T_n\setMS)  \le  \widetilde{C}_1 \left ( \lambda_{n}^{2} \, \sup_{s\in \setMS} \| s \|_{L^1}   + 2 \lambda_{n} V_a^b(F) \right ) 
 \le \widetilde{C}_1 \left ( \lambda_{n}^{2} \, \|F\|_\infty (b-a)   + 2 \lambda_{n} V_a^b(F) \right ).
$$
If, in addition,  $T_n e_0 \equiv e_0$, then
$$
	\haus(\setMS , T_n\setMS)  \le  \widetilde{C}_2 \left ( \lambda_{n}^{2} \, (b-a)  + 2 \lambda_{n}  \right ) V_a^b(F) .
$$
Here $\widetilde{C}_i$, $i=0,1$ depend only on $[a,b]$ and on the underlying norm in $\Rd$.
\end{theorem}

\subsection{Examples}\label{Sect_Examples_in_$L^1$}

We apply the results of Section~\ref{Sect_Error bounds_in_$L^1$} to the Bernstein-Durrmeyer operator $M_n$ and the Kantorovich operator $K_n$. 
As was shown in Section~\ref{Sect_SpecificOperators}, these operators satisfy $ M_n e_0 \equiv e_0$ , $K_n e_0 \equiv e_0 $.
Moreover,  
$$
M_n e_1(x) - e_1(x) = \frac{nx+1}{n+2} - x =  \frac{1-2x}{n+2}, \quad K_n e_1(x)  - e_1(x) = \frac{2nx+1}{2(n+1)} - x =  \frac{1-2x}{2(n+1)}\, .
$$ 
Combining this with the second claim of Theorem~\ref{Theorem_ErrorBounds_L1} we end up with the following result.

\begin{theorem}\label{Theorem_ConvergenceRate_Kant}
	Let $F \in \mathcal{F}[0,1]$ and let $T_n$  be the Bernstein-Durrmeyer operator or the Kantorovich operator.  Then 
$$
	\haus(\setMS , T_n\setMS) \le C V_a^b(F)\left ( \lambda_{n}^{2} \, (b-a)+ 2 \lambda_{n} \right ) = O \left( \frac{1}{\sqrt{n}} \right),
\quad  n \to \infty,
$$
with $C$ depending only on the underlying norm in $\Rd$.
\end{theorem}
Note that the rate $O \left( \frac{1}{\sqrt{n}} \right)$  here is obtained for any $F \in \mathcal{F}[0,1]$, while in the pointwise estimates of Section~\ref{Section_Rate_Pointwise} we obtain this rate under the assumption of 
 local Lipschitz continuity of $F \in \mathcal{F}[0,1]$.

 


\end{document}